\documentclass[conference]{article}
\usepackage{cite}
\usepackage{amsmath,amssymb,amsfonts}
\usepackage{amsthm}

\usepackage{algorithmic}
\usepackage{graphicx}
\usepackage{textcomp}
\usepackage{tikz-cd}
\usepackage{xcolor}
\usepackage[dvipsnames]{xcolor}
\usepackage{standalone}

\newtheorem{theorem}{Theorem}
\newtheorem{lemma}{Lemma}
\newtheorem{corollary}{Corollary}

\newtheorem{proposition}{Proposition}

\newtheorem{definition}{Definition}

\def\BibTeX{{\rm B\kern-.05em{\sc i\kern-.025em b}\kern-.08em
    T\kern-.1667em\lower.7ex\hbox{E}\kern-.125emX}}

\DeclareRobustCommand{\stirling}{\genfrac\{\}{0pt}{}}

\newcommand{\ck}[1]{}   
\newcommand{\bw}[1]{}   

\newcommand{\nCol}{k}

\newcommand{\pkgSet}{P}               
\newcommand{\sourcePkg}{p_0}          
\newcommand{\poSet}{\mathcal{P}}               
\newcommand{\dpoSet}[1]{\mathcal{P}^{\downarrow}_{#1}}               
\newcommand{\upoSet}[1]{\mathcal{P}^{\uparrow}_{#1}}               

\newcommand{\childrenOf}[1]{\operatorname{child}(#1)}     

\newcommand{\canonicalOrder}{\geq}    

\newcommand{\downset}[1]{\ensuremath{\downarrow #1}}
\newcommand{\upset}[1]{\ensuremath{\uparrow #1}}

\newcommand{\maxAntichains}[1]{\mathcal{S}(#1)}

\newcommand{\inducedPoSet}[1]{\poSet_{#1}}

\newcommand{\maxDSC}[1]{\mathcal{C}_\nCol(#1)}

\newcommand{\colourSet}{\mathcal{K}}

\begin{document}

\title{Bounds on Decorated Sweep Covers in Tree Posets\\
}

\author{
    Blake A. Wilson \\
    Elmore Family School of ECE \\
    Purdue University \\
    West Lafayette, IN \\
    0000-0002-3055-8334
    \and
    Colin Krawchuk \\
    Mathematics\\
    University of Cambridge\\
    Cambridge, UK
}

\maketitle

\ck{Scope of this work: Bounds on coloured sweep covers. I would say this rather than maximal antichains, as we don't cover that in full generality. If you are comfortable defining a sweep cover to be a maximal antichain in a tree, then I think this better fits what we actually do here. }
\begin{abstract}
We introduce decorated sweep covers as a colouring on maximal antichains in tree posets such that if two elements have the same colour they are siblings. 
DSCs appear in applications wherever maximal antichains require structural differentiation among parallel options that have a common ancestry, e.g., distributed systems, drone routing in logistics, and Monte Carlo Tree Search. 
We restrict our analysis to enumerating $k$-coloured DSCs in $n$-ary tree posets and prove i) their ordinary generating function in Theorem \ref{thm:ogf}, ii) new Schur-convexity results for binomial coefficients in Theorem \ref{thm:bin} and iii) bounds on the OGF coefficients which scale as $\Theta(D_n^k k^{-3/2})$ in Theorem \ref{thm:bounds} where $D_n > n$ is the exponential growth constant for each $n$.

\ck{In the abstract where you list the achievements, specify the result/section and reference it. (ie. In theorem ... we prove ...}
\bw{Done}
\end{abstract}

\section{Introduction}

The study of antichains in partially ordered sets (posets) is central to many topics in analysis of algorithms and enumerative combinatorics, with connections to extremal set theory, distributed computing, and decision-making in structured environments \cite{erdos2007splitting, gellert2018antichains, tsai2017maxantichains}. 
A maximal antichain—a set of mutually incomparable elements that are comparable to the rest of the poset—plays a key role in understanding the computational complexity of many algorithms involving posets. 

In many applications, such as task scheduling in distributed systems or multi-agent planning in tree-like environments \cite{he2023parallelism, bosek2018chainpartition}, it is useful to assign additional structure to antichains, such as colouring elements to encode constraints or parallelism.
In this work, we introduce and study \emph{decorated sweep-covers} (DSCs): colourings of maximal antichains in which only sibling elements (i.e., elements sharing a parent) may receive the same colour.
Our main goal is to enumerate DSCs in tree posets \cite{jiang2024treeposets}, to understand how they grow with the number of colours and the structure of the underlying tree.
We develop recursive constructions and derive asymptotic bounds.

The main contributions of this paper are:
1) We formalize DSCs and provide a decomposition theorem for their structure in tree posets.
2) We establish log-convexity and obtain sharp bounds for binomial products in their upper parameter, using tools from analytic combinatorics \cite{flajolet2009analytic}.
3) We derive recursive and explicit formulas for the number of DSCs in infinite $n$-ary trees, and analyze the asymptotic growth of this quantity.
Namely, we prove the following results:
\begin{theorem}[Simplified]
    The number of lower $\nCol$-DSCs in an infinite, $n$-ary tree poset is generated by the ordinary generating function (OGF) 
\[
F_n(z) = \sum_{u=0}^n \binom{n}{u} B_u(z) F_n(z)^{n - u},
\]
where $B_u(z) = \sum_{k=0}^u \stirling{u}{k} z^k$ is the $u$-th Touchard polynomial.
 \label{thm:ogf}
\end{theorem}
Lower DSCs account for all except one DSC in $T_n^\infty$, so the total quantity of DSCs is given by $F_n(z) + 1$.
In our case the Bell numbers which construct $B_u(z)$ give rise to an OGF in Thm. \ref{thm:ogf}, which cannot be solved directly due to Abel-Ruffini.
Therefore we
use bootstrapping and the following new theorem 
to show bounds in Thm. \ref{thm:bounds}.
\begin{theorem}
    
Let $N = (n_1, \dots, n_m)$ be an integer composition (ordered partition) of $n = \sum_{j=1}^m n_j$ with each $n_j \ge k$, and let $k \ge 1$. 
Then:
\[
\left( \frac{n}{ekm} \right)^{kn} \left(\sqrt{2\pi n/m} \right)^k
\leq
\prod_{n_j\in N} \prod_{i = 1}^{n_j} \binom{i}{k}
\leq
\left( \frac{\tilde{n}}{k} \right)^{k \tilde{n} + \frac{k}{2} - \frac{\tilde{n}}{2}} (\sqrt{2\pi})^{k - \tilde{n}},
\]
where $\tilde{n} = n - (m - 1)k$.
\label{thm:bin}
\end{theorem}

\begin{theorem}
The coefficients $f_n(\nCol)$ for $F_n(z)$ exhibit purely exponential growth with polynomial modifiers, bounded by:
\begin{align}
    f_n(\nCol) &= \Theta(D_n^k k^{-3/2}) \nonumber
\end{align} \label{thm:bounds}
where $D_n > n$.
\end{theorem}
While the result in Thm. \ref{thm:bin} is straightforward from Schur-convexity and Stirling's approximation, the authors are not aware of an existing proof in the literature.

\section{Preliminaries}

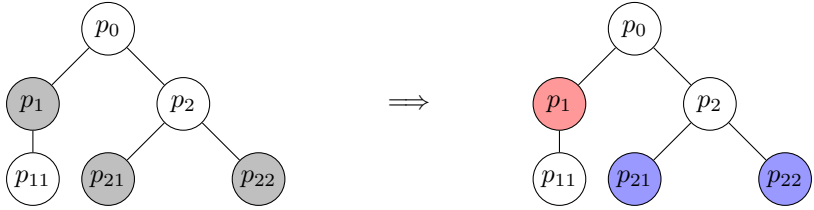
\begin{figure}[ht]
  \centering
  \begin{tikzpicture}[
      level distance=10mm,
      sibling distance=20mm,
      every node/.style={circle, draw, minimum size=7mm, inner sep=1pt},
      antichain/.style={fill=gray!50},
      colA/.style={fill=red!40},
      colB/.style={fill=blue!40}
    ]
    \begin{scope}
      \node (p0) {$p_0$}
        child { node[antichain] (p1) {$p_1$}
          child { node (p11) {$p_{11}$} }
        }
        child { node (p2) {$p_2$}
          child { node[antichain] (p21) {$p_{21}$} }
          child { node[antichain] (p22) {$p_{22}$} }
        };
    \end{scope}

    \begin{scope}[xshift=5cm]
      \node[draw=none] at (-1,-1) {\(\Longrightarrow\)};
    \end{scope}

    \begin{scope}[xshift=7cm]
      \node (q0) {$p_0$}
        child { node[colA] (q1) {$p_1$}
          child { node (q11) {$p_{11}$} }
        }
        child { node (q2) {$p_2$}
          child { node[colB] (q21) {$p_{21}$} }
          child { node[colB] (q22) {$p_{22}$} }
        };
    \end{scope}
  \end{tikzpicture}
  \caption{A 2‐level tree poset: on the left the maximal antichain $S = \{p_{1}, p_{21}, p_{22}\}$ is shaded grey; on the right the same maximal antichain is coloured as a DSC $c(p_{1}) = \text{red}$, $c(p_{21}) = c(p_{22}) = \text{blue}$.}
  \label{fig:scma_example}
\end{figure}

A \emph{partially ordered set} (or \emph{poset}) $\poSet = (\pkgSet,\canonicalOrder)$ consists of a set $\pkgSet$ and a reflexive, antisymmetric, and transitive binary relation $\canonicalOrder$. All posets considered in this work will be bounded above. That is, $\poSet$ has a maximum element $\sourcePkg\in\pkgSet$, satisfying $\sourcePkg \canonicalOrder p $ for all $p\in\pkgSet$.
A subset $S\subseteq\pkgSet$ is an \emph{antichain} if its elements are pairwise incomparable. An antichain is \emph{maximal} if every element not in the antichain is comparable to at least one element in the antichain. An element $p$ is said to cover $q$ if $q < p$ and no element $r$ satisfies $q < r < p$. The \emph{Hasse diagram} of a poset is the directed graph representing the cover relation: there is an edge $(p,q)$ if $q$ covers $p$.

For any element $p\in\pkgSet$, we refer to the elements that cover $p$ as its \emph{parents}, the elements covered by $p$ as its \emph{children}, and the elements sharing at least one common parent with $p$ as its \emph{siblings}. This paper studies \emph{decorated sweep covers} (DSCs), defined as follows:

\begin{definition}[Decorated Sweep Cover (DSC)]
Given a poset $\poSet$ and a colour set $\colourSet$ with $\nCol$ colours, a \emph{decorated sweep cover (DSC)} is a colouring $c:S\to\colourSet$ of a maximal antichain $S\subseteq\pkgSet$ where $c(p)=c(q)$ implies $p$ and $q$ are siblings.
\end{definition}
We denote by $\maxDSC{\poSet}$  the set of all DSCs in $\poSet$ using exactly $k$ colours, defined up to relabeling (permutations of the colours). Thus, the DSC enumeration is defined modulo the symmetric group action on the colours. 
In this work, we investigate the asymptotic growth of DSCs and analyze their combinatorial complexity. The motivation is to provide algorithmic analysis for problems in distributed systems and pursuit–evasion settings.

\section{Methods}
In this section, we lay out the necessary lemmas for proving Thm. \ref{thm:ogf} and Thm. \ref{thm:bounds}. 
For Thm. \ref{thm:ogf}, we mainly use the disjointness of induced posets at every element in a tree poset to construct a recurrence relation.
From the recurrence, the OGF falls out naturally.
The proof of Thm. \ref{thm:bounds}, relies on Thm. \ref{thm:bin}, which is proven in Appendix \ref{app:bin}.
\subsection{OGF for DSCs in Infinite $n$-ary Trees}

Denote by $\maxAntichains{\poSet}$ the set of maximal antichains of $\poSet$.
The down-set of an element $p$ is defined as $\downset{p}:= \{ q \in P: q \leq p \}$ and similarly, the up-set of $p$ is $\upset{p}:= \{ q \in P: q \geq p \}$.
We say $\dpoSet{p} = (\downset{p}, \canonicalOrder)$ is the down-induced poset of $p$, and $\upoSet{p} = (\upset{p}, \canonicalOrder)$ for the up-induced poset of $p$. 
Let $\childrenOf{p}=\{p_1,\dots,p_m\}$ denote the children of an element $p$.
We say a poset $\poSet$ is disjoint at an element $p$ iff the down-sets $\{\downset{p_1},...,\downset{p_n}\}$ given by the children $\{p_1,\dots,p_n\}$ of $p$ are all mutually disjoint. 
This property holds at every element that has children in a tree poset and we exploit it to derive a recurrence relation based on the regularity of infinite $n$-ary trees.

\begin{corollary}
    Any tree poset $\poSet$ is disjoint at every element $p$. \label{cor:tree_disjoint}
\end{corollary}

\begin{lemma}[Enumeration via Disjointness]\label{lem:sigma_bij}
Let $\poSet_1,\dots,\poSet_n$ be pairwise disjoint posets. Define
\[
\sigma:\maxAntichains{\poSet_1}\times\cdots\times\maxAntichains{\poSet_n}\to\maxAntichains{\bigcup_i\dpoSet{p_i}},\quad\sigma(\pi_1,\dots,\pi_n)=\bigcup_i\pi_i.
\]
Then, $\sigma$ is a bijection, and thus
\[
|\maxAntichains{\cup_i\dpoSet{p_i}}|=\prod_{i=1}^n|\maxAntichains{\dpoSet{p_i}}|.
\]
\end{lemma}

\begin{proof}
Injectivity is immediate since each $\dpoSet{p_i}$ is disjoint. Surjectivity follows from every maximal antichain in the union necessarily and uniquely decomposing into maximal antichains within each individual poset.
\end{proof}

\begin{corollary}\label{cor:induced_count}
If $\poSet$ is disjoint at $p$, then
\[
|\maxAntichains{\inducedPoSet{p}}|=1+\prod_{p_i\in\childrenOf{p}}|\maxAntichains{\inducedPoSet{p_i}}|.
\]
\end{corollary}

\begin{proof}
Immediate by considering either the trivial antichain $\{p\}$ or recursively counting maximal antichains formed by excluding $p$.
\end{proof}


Now, extending to DSCs in trees, each maximal antichain $S$ admits a unique decomposition into collections of siblings: $Y(S)=\{Y_1,\dots,Y_m\}$.

\begin{lemma}[DSC Decomposition]\label{lem:DSC_decomp}
Every DSC colouring $c:S\to\colourSet$ uniquely decomposes as
\[
c=\bigcup_{i=1}^m c_i,\quad c_i:Y_i\to\colourSet_i,\quad\colourSet_i\cap\colourSet_j=\emptyset\text{ for }i\ne j.
\]
\end{lemma}

\begin{proof}
Follows directly from the uniqueness and disjointness of maximal sibling groups.
\end{proof}

\begin{lemma}[Counting DSCs Explicitly]\label{lem:kDSC_from_S}
Given $S$ with sibling decomposition $Y(S)=\{Y_1,\dots,Y_m\}$, the number of distinct $\nCol$-coloured DSCs on $S$ is
\[
C_{\nCol}(S)=\sum_{\substack{k_1+\dots+k_m=\nCol\\k_i\ge1}}\prod_{i=1}^m\stirling{|Y_i|}{k_i},
\]
where $\stirling{n}{k}$ denotes the Stirling number of the second kind.
\end{lemma}

\begin{proof}
Each sibling set $Y_i$ must be coloured distinctly from other sibling sets. For each sibling set $Y_i$, we have exactly $\stirling{|Y_i|}{k_i}$ ways to partition $|Y_i|$ elements into $k_i$ colour classes. Summing over all valid compositions $k_1+\dots+k_m=\nCol$ gives the total.
\end{proof}

\begin{lemma}[Base Case: One Colour]\label{lem:base_case_k1}
For any non-maximal element $p \in \poSet$, i.e., $\childrenOf{p} \ne \emptyset$, in a tree poset $\poSet$, there are exactly two 1-coloured DSCs of $\dpoSet{p}$:
\[
S=\{p\}\quad\text{and}\quad S=\childrenOf{p}
\].
Otherwise, $p$ is a maximal element, i.e., $\childrenOf{p} = \emptyset$, and has a single 1-coloured DSC of $S = \{p\}$. 
\end{lemma}

\begin{proof}
With only one colour available, elements must be siblings. Hence, these two maximal antichains are the only possibilities. If the down-set is empty then it's not 1-colourable so we remove it.
\end{proof}
\begin{lemma}[DSC Compositionality]\label{lem:compose_DSC}
Consider a tree poset $\poSet$ and an element $p$ with down-induced poset $\dpoSet{p}$. Then, the number of $\nCol$-coloured DSCs in $\dpoSet{p}$ decomposes as:
\[
|\maxDSC{\dpoSet{p}}| 
= \sum_{\substack{k_1+\dots+k_m=\nCol\\k_i\ge1}} \prod_{i=1}^m\sum_{S\in\maxAntichains{\dpoSet{p_i}}}C_{k_i}(S)
= \sum_{\substack{k_1+\dots+k_m=\nCol\\k_i\ge1}} \prod_{i=1}^m|\mathcal{C}_{k_i}(\dpoSet{p_i})|
\]
\end{lemma}

\begin{proof}
First, observe that by Lemma \ref{lem:sigma_bij} any maximal antichain $S\in\maxAntichains{\dpoSet{p}}$ decomposes uniquely into maximal antichains of each disjoint poset:
\[
S=\bigcup_{i=1}^m S_i,\quad\text{with }S_i\in\maxAntichains{\dpoSet{p_i}} \text{ where } p_i \in \childrenOf{p}.
\]

By Lemma~\ref{lem:kDSC_from_S}, for a fixed antichain $S$ with sibling decomposition $Y(S)=\{Y_j\}$, the number of distinct $\nCol$-coloured DSCs is given by
\[
C_{\nCol}(S)=\sum_{\substack{k_1+\dots+k_{|Y(S)|}=\nCol\\k_j\ge1}}\prod_{j=1}^{|Y(S)|}\stirling{|Y_j|}{k_j}.
\]

Since the posets $\{\dpoSet{p_i} : p_i \in \childrenOf{p}\}$ are disjoint, sibling sets in $S$ are also partitioned according to these posets.
Thus, the DSC counts factor across the decomposition into disjoint posets. 
Specifically, each $C_{\nCol}(S)$ decomposes as follows:
\[
C_{\nCol}(S) 
= \sum_{\substack{k_1+\dots+k_m=\nCol\\k_i\ge1}} \prod_{i=1}^m C_{k_i}(S_i),
\quad\text{where }S_i\in\maxAntichains{\dpoSet{p_i}}.
\]

Summing over all possible maximal antichains $S\in\maxAntichains{\poSet}$, we get:
\[
|\maxDSC{\poSet}| 
= \sum_{S\in\maxAntichains{\poSet}}C_{\nCol}(S)
= \sum_{S_1\in\maxAntichains{\poSet_1}}\dots\sum_{S_m\in\maxAntichains{\poSet_m}}
  \sum_{\substack{k_1+\dots+k_m=\nCol\\k_i\ge1}}\prod_{i=1}^m C_{k_i}(S_i).
\]

Interchanging the order of summation, we collect by colours first:
\[
|\maxDSC{\poSet}| 
= \sum_{\substack{k_1+\dots+k_m=\nCol\\k_i\ge1}}
\prod_{i=1}^m\sum_{S\in\maxAntichains{\dpoSet{p_i}}}C_{k_i}(S).
\]

Finally, note by definition that:
\[
|\mathcal{C}_{k_i}(\dpoSet{p_i})|=\sum_{S\in\maxAntichains{\dpoSet{p_i}}}C_{k_i}(S),
\]
thus obtaining the stated result:
\[
|\maxDSC{\poSet}| 
= \sum_{\substack{k_1+\dots+k_m=\nCol\\k_i\ge1}}\prod_{i=1}^m|\mathcal{C}_{k_i}(\dpoSet{p_i})|.
\]
\end{proof}




\begin{lemma}[Recursive Enumeration of DSCs]\label{lem:recursive_DSC_at_p}
Let $\poSet$ be a poset disjoint at $p$. Then the number of $\nCol$-coloured DSCs at $\inducedPoSet{p}$ is given recursively by:
\[
|\maxDSC{\inducedPoSet{p}}|=\sum_{U\subseteq\childrenOf{p}}\sum_{\nCol'=|U|}^{\nCol}|\mathcal{C}_{\nCol'}(U)|\sum_{\substack{k_1+\dots+k_{m-|U|}=\nCol-\nCol'\\k_j\ge1}}\prod_{p_j\notin U}|
\mathcal{C}_{k_j}({\inducedPoSet{p_j}-\{p_j\}})|.
\]
\end{lemma}

\begin{proof}
We start by noting that for a poset $\poSet$ disjoint at $p$, each maximal antichain excluding $\{p\}$ (i.e., $S \in \maxAntichains{\inducedPoSet{p}} - \{p\}$) is partitioned via 
\[\pi(S) = \{S \cap \inducedPoSet{p_i} : p_i \in \childrenOf{p}\}\]
into a unique  \emph{upper} set $U(S) \subseteq \childrenOf{p}$ and a unique \emph{lower} set $L(S) = \{S_j \in \maxAntichains{\inducedPoSet{p_j}-\{p_j\}} : p_j\notin U\}$.
By Lemma~\ref{lem:compose_DSC}, the number of $\nCol$-DSCs decomposes into colouring $U(S)$ and $L(S)$ independently, subject to colour-partition constraints.
Enumerate all subsets $U$ that select elements directly below $p$, forming upper DSCs.
Then, assign colours to $U$ in exactly $|\mathcal{C}_{\nCol'}(U)|$ ways.
Finally, independently assign remaining colours across maximal antichains in lower subtrees, utilizing compositionality across disjoint posets (Lemma~\ref{lem:compose_DSC}).
\end{proof}

\subsubsection{Proof of Thm.~\ref{thm:ogf}}

We consider a poset $\poSet$ whose Hasse diagram $H(\poSet)=\mathcal{T}_n^\infty$ is an infinite $n$-ary tree. For clarity, define $f_n(k)$ as the number of lower-DSCs using exactly $k$ colours.  
Naturally, $f_n(1) = 1$ by evaluating $k=1$ via Lemma \ref{lem:compose_DSC}.
We first derive a recurrence for $f_n(k)$:

\begin{lemma}[DSC Recurrence Relation]\label{lem:recurrence}
The number $f_n(k)$ of DSCs with exactly $k$ colours in an infinite $n$-ary tree satisfies the recurrence:
\[
f_n(k)=\sum_{u=0}^{n}\sum_{k'=0}^u\binom{n}{u}\stirling{u}{k'}\sum_{\substack{k_1+\dots+k_{n-u}=k-k'\\k_j\ge1}}\prod_{j=1}^{n-u}f_n(k_j),
\]
where $\stirling{n}{k}$ is a Stirling number of the second kind.
\end{lemma}

\begin{proof}
Fix an internal node $p$ with children $\{p_1,\dots,p_n\}$. By Lemma~\ref{lem:recursive_DSC_at_p}, DSCs at $\inducedPoSet{p}$ can be enumerated by considering all partitionings into upper and lower sets $U(S)$ and $L(S)$. 
The case $U(S) = \childrenOf{p}$ (where $u=n$) contributes $\stirling{n}{k}$, as the sum over products for the $n-u=0$ remaining subtrees evaluates to $1$ (following the convention that an empty product is $1$ and the empty sum constraint forces $k'=k$).
For $U(S) \subseteq \childrenOf{p}$, exactly $u$ children are directly included in the antichain. There are $\binom{n}{u}$ ways to select these $u$ children and $\stirling{u}{k'}$ ways to colour them with exactly $k'$ colours. The remaining $(n-u)$ subtrees rooted at the children excluded from the antichain are independently coloured. Letting these colours sum to $(k - k')$ across the subtrees, we have:
    \[
    \sum_{\substack{k_1+\dots+k_{n-u}=k-k'\\k_j\ge1}}\prod_{j=1}^{n-u}f_n(k_j).
    \]

Summing these contributions over all possible $u$ and colour partitions yields the stated recurrence.
\end{proof}

Now, we leverage this recurrence to derive the ordinary generating function (OGF):

\begin{proof}[Proof of Theorem \ref{thm:ogf}]
Define the OGF as
\[
F_n(z)=\sum_{k=0}^{\infty}f_n(k)z^k,\quad\text{with}\quad B_n(z)=\sum_{k=0}^{n}\stirling{n}{k}z^k.
\]

Using the same sum and product conventions as in Lemma~\ref{lem:recurrence}:
\[
f_n(k)=\sum_{u=0}^{n}\sum_{k'=0}^{u}\binom{n}{u}\stirling{u}{k'}\sum_{\substack{k_1+\dots+k_{n-u}=k-k'\\k_j\ge1}}\prod_{j=1}^{n-u}f_n(k_j).
\]

Multiplying by $z^k$ and summing over all $k\ge0$, we obtain the generating function:
\begin{align*}
F_n(z)&=\sum_{u=0}^{n}\binom{n}{u}\sum_{k'=0}^{u}\stirling{u}{k'}z^{k'}\sum_{k\ge k'}\sum_{\substack{k_1+\dots+k_{n-u}=k-k'\\k_j\ge1}}z^{k-k'}\prod_{j=1}^{n-u}f_n(k_j).
\end{align*}

Interchange summations and factor out terms dependent on $k'$:
\[
\sum_{k'=0}^{u}\stirling{u}{k'}z^{k'}=B_u(z).
\]

Observe the inner sum is a convolution of generating functions, which becomes:
\[
\sum_{k\ge k'}\sum_{\substack{k_1+\dots+k_{n-u}=k-k'\\k_j\ge1}}z^{k-k'}\prod_{j=1}^{n-u}f_n(k_j)=F_n(z)^{n-u}.
\]

Thus, the OGF becomes:
\[
F_n(z)=\sum_{u=0}^{n}\binom{n}{u}B_u(z)F_n(z)^{n-u}.
\]

\end{proof}

\subsection{Proof of Thm. \ref{thm:bounds}}

\begin{proposition}\label{prop:implicit}
For fixed $n \ge 2$, the generating function $y = F_n(z)$ is implicitly defined by the bivariate polynomial equation $y = \Phi(z, y)$, where
$$
\Phi(z, y) = \sum_{u=0}^n \binom{n}{u} B_u(z) y^{n - u}
$$
Furthermore, $F_n(z)$ has a dominant square-root singularity $R_n > 0$.
\end{proposition}

\begin{proof}
Because the Touchard polynomials $B_u(z)$ are polynomials in $z$, the function $\Phi(z, y)$ is a bivariate polynomial. The equation $y = \Phi(z, y)$ demonstrates that $F_n(z)$ is an algebraic function. 

Since the coefficients of $F_n(z)$ enumerate DSCs, they are strictly non-negative, meaning $\Phi(z, y)$ has non-negative coefficients. By the Smooth Implicit-Function Schema (Theorem VII.3 in Flajolet and Sedgewick \cite{flajolet2009analytic}), the solution $y = F_n(z)$ possesses a dominant singularity $R_n > 0$ on the positive real axis. This singularity is a branch point of order 2, meaning the local expansion of $F_n(z)$ near $R_n$ takes the form:
$$
F_n(z) = y_0 - c \sqrt{1 - \frac{z}{R_n}} + O\left(1 - \frac{z}{R_n}\right)
$$
for some constants $y_0 > 0$ and $c > 0$.
\end{proof}

\begin{lemma}[Lower bound for $f_n(k)$]\label{lem:lower_bound}
For every fixed $n \ge 2$, there exists a constant $A_n > 0$ such that
$$
f_n(k) \ge A_n n^k k^{n^2-1} \quad (\forall k \ge n).
$$
\end{lemma}

\begin{proof}
We split the recurrence from Lemma \ref{lem:recurrence} into two principal pieces:
$$
f_n(k) \ge f_n^{(a)}(k) + f_n^{(b)}(k),
$$
where $f_n^{(a)}(k)$ is the term corresponding to $u=n-1$, and $f_n^{(b)}(k)$ is the term where $u=0$:
$$
f_n^{(b)}(k) := \sum_{\substack{k_1+\dots+k_n=k\\k_j\ge1}} \prod_{j=1}^n f_n(k_j).
$$
The evaluation of the $k'=1$ sub-case in $f_n^{(a)}(k)$ yields $f_n^{(a)}(k) \ge n f_n(k-1)$, which immediately implies $f_n(k) \ge E_n n^k$ for some constant $E_n > 0$ and all $k \ge n$. Inserting this base exponential bound into $f_n^{(b)}(k)$, we obtain:
$$
f_n^{(b)}(k) \ge E_n^n n^k \sum_{\substack{k_1+\dots+k_n=k\\k_j\ge1}} \prod_{j=1}^n k_j^{n-1}.
$$
The inner sum over the composition is exactly the form evaluated in Lemma \ref{lem:powered-composition} (Appendix \ref{app:bin}), which scales as $\Omega(k^{n^2-1})$. Thus, $f_n^{(b)}(k) \ge C_n n^k k^{n^2-1}$.
The exponent $n^2-1$ in the $k$-factor strictly dominates any polynomial contribution originating from $f_n^{(a)}(k)$, ensuring the combined lower bound holds asymptotically for all sufficiently large $k$.
\end{proof}

\begin{proposition}[Analytic Asymptotic Growth]\label{prop:analytic_upper}
Let $R_n$ denote the dominant radius of convergence of the ordinary generating function $F_n(z) = \sum_{k=0}^\infty f_n(k)z^k$. There exists a constant $D_n = 1/R_n > 0$ such that:
$$
f_n(k) \sim \frac{c}{2\sqrt{\pi}} k^{-3/2} D_n^k.
$$
\end{proposition}

\begin{proof}
By Proposition \ref{prop:implicit}, the generating function $F_n(z)$ has a dominant square-root singularity at $z = R_n$, with the local singular expansion being $-c\sqrt{1 - z/R_n}$. By the Transfer Theorem for algebraic functions (Theorem VII.8 in \cite{flajolet2009analytic}), the asymptotic behavior of the coefficients $f_n(k)$ is determined directly by this singular term.

We translate the local expansion $c(1 - z/R_n)^{1/2}$ into an asymptotic equivalence for the coefficients:
$$
f_n(k) \sim [z^k] \left( -c \left(1 - \frac{z}{R_n}\right)^{1/2} \right).
$$
Applying the standard asymptotic expansion for binomial coefficients $\binom{1/2}{k}$, we obtain the exact polynomial modifier $\beta = -3/2$:
$$
f_n(k) \sim \frac{-c}{\Gamma(-1/2)} k^{-3/2} R_n^{-k}.
$$
Since $\Gamma(-1/2) = -2\sqrt{\pi}$ and defining the exponential growth constant as $D_n = 1/R_n$, the signs cancel out, yielding the strict asymptotic:
$$
f_n(k) \sim \frac{c}{2\sqrt{\pi}} k^{-3/2} D_n^k.
$$
\end{proof}

\begin{proof}[Proof of Theorem \ref{thm:bounds}]
From Lemma \ref{lem:lower_bound}, we established the polynomial lower bound $f_n(k) \ge A_n n^k k^{n^2-1}$ for all $k \ge n$. From Proposition \ref{prop:analytic_upper}, the exact asymptotic growth is established as $f_n(k) \sim C_n k^{-3/2} D_n^k$, where $C_n = \frac{c}{2\sqrt{\pi}}$ and $D_n = 1/R_n$.

There exists a constant $B_n > 0$ such that for all sufficiently large $k$, we can bound the sequence above by:
$$
f_n(k) \le B_n k^{-3/2} D_n^k.
$$

Chaining the lower and upper bounds together, for sufficiently large $k$, we must have:
$$
A_n n^k k^{n^2-1} \le B_n k^{-3/2} D_n^k.
$$

Rearranging the terms to isolate the exponential and polynomial bases yields:
$$
\left(\frac{D_n}{n}\right)^k \ge \frac{A_n}{B_n} k^{n^2 + 1/2}.
$$
Because $n \ge 2$, the exponent $n^2 + 1/2$ is strictly positive, meaning the right-hand side is a polynomial that diverges to infinity as $k \to \infty$. For the inequality to hold asymptotically, the exponential base on the left-hand side must be strictly greater than 1. Therefore, it is required that $D_n > n$.
Thus, the sequence exhibits purely exponential growth with a polynomial modifier of $\beta = -3/2$, giving:
$$
f_n(k) = \Theta(D_n^k k^{-3/2}),
$$
where $D_n > n$, completing the proof.
\end{proof}

\subsection{Numerical Results}

In Figure \ref{fig:asymptotics}, we provide numerical convergence plots supporting our analysis.

\begin{figure*}[t]
    \centering
    \noindent\makebox[\textwidth][c]{%
        \includegraphics[width=\textwidth]{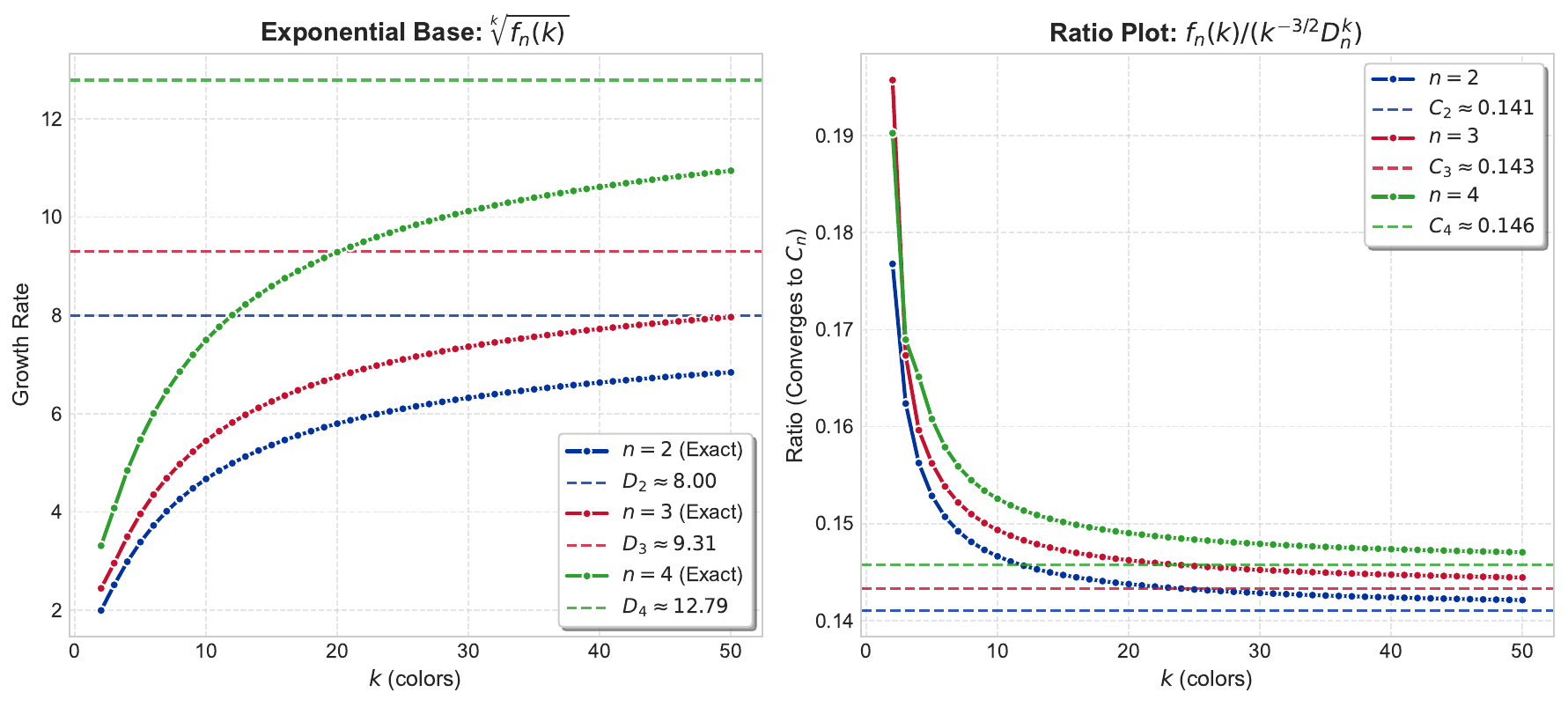}%
    }
    \caption{Empirical verification of the asymptotic bounds for Decorated Sweep Covers in $n$-ary trees. 
    \textbf{Left:} The $k$-th root of the sequence converges to the exponential base $D_n$. The crossing of the dotted structural branching factor lines empirically confirms $D_n > n$. 
    \textbf{Right:} The ratio of the exact sequence $f_n(k)$ to its theoretical shape $k^{-3/2}D_n^k$. The strict horizontal convergence to $C_n$ numerically supports the asymptotic expansion $f_n(k) \sim C_n k^{-3/2} D_n^k$.}
    \label{fig:asymptotics}
\end{figure*}

\section{Conclusion}

In this paper, we introduced and formalized decorated sweep-covers (DSCs) to capture the structured colouring of maximal antichains.
Motivated by the natural constraints of distributed computing, multi-agent planning, and pursuit-evasion scenarios, we provided a rigorous combinatorial framework for analyzing these structures. 
By exploiting the inherent disjointness of tree posets, we established a decomposition theorem that reduces the complex global colouring of antichains into independent, compositional sub-problems.

Our primary enumerative achievement is the derivation of the ordinary generating function for the number of $k$-coloured lower DSCs in an infinite $n$-ary tree, defined implicitly via Touchard polynomials:
$$F_n(z) = \sum_{u=0}^n \binom{n}{u} B_u(z) F_n(z)^{n - u}$$
Because exact coefficient extraction is obstructed by the algebraic complexity of this functional equation, we utilized singularity analysis techniques from analytic combinatorics to deduce the asymptotic behavior of the sequence.
We proved that the coefficients $f_n(k)$ exhibit purely exponential growth modified by a strict polynomial factor, yielding $f_n(k) = \Theta(D_n^k k^{-3/2})$, where the exponential base $D_n$ strictly exceeds the structural branching factor $n$. To support this asymptotic analysis, we also introduced a novel bounding theorem for binomial products over integer compositions.

\subsection{Directions for Future Research}
The theoretical foundation laid in this work opens several avenues for both combinatorial and applied extensions:
While our analysis focused on infinite regular $n$-ary trees, extending this framework to finite bounded trees, asymmetric trees, or random Galton-Watson trees would provide a more generalized view of DSC growth in unpredictable environments.
The exponential growth constant $D_n = 1/R_n$ is currently defined implicitly by the dominant singularity of $F_n(z)$. Further algebraic investigation into this singularity could yield closed-form approximations or exact limits for $D_n$ as $n \to \infty$.
With the asymptotic combinatorial complexity of DSCs established, future work can directly apply these bounds to optimize search space pruning in distributed scheduling algorithms and limit the state-space in bounded packaging games. 

\pagebreak

\section{Appendix}
\subsection{Bounds on Compositions}\label{app:bin}

\begin{lemma}[Polynomial lower bound for powered composition products]\label{lem:powered-composition}
Let
\[
S_{n}(k)=\sum_{\substack{k_{1}+\dots +k_{n}=k\\ k_{j}\ge 1}} \bigl(k_{1}k_{2}\cdots k_{n}\bigr)^{n},\qquad n\ge 1,\;k\ge n.
\]
Then there exists a constant $c_{n}>0$ (depending only on $n$) such that
\[
S_{n}(k) \ge c_{n} k^{n(n+1)-1} \quad\text{for all }k\ge n.
\]
\end{lemma}

\begin{proof}
For a single summand $k_{j}^n$, we use the ordinary generating function (OGF):
\[
A_{n}(x)=\sum_{m\ge1} m^{n}x^{m}=\frac{x P_{n}(x)}{(1-x)^{n+1}},
\]
where $P_{n}(x)$ is a polynomial of degree $n-1$ whose coefficients are the Eulerian numbers. Because the $n$ parts are independent, the OGF of the convolution $S_n(k)$ is the Cauchy product:
\[
\bigl[A_{n}(x)\bigr]^{n}=\frac{x^{n} P_{n}(x)^{n}}{(1-x)^{n(n+1)}}.
\]
Writing $P_{n}(x)^{n}=\sum_{j=0}^{n(n-1)}a_{n,j}x^{j}$ and extracting the coefficient of $x^{k}$ using the negative binomial series identity gives:
\[
S_{n}(k)=\sum_{j=0}^{n(n-1)}a_{n,j}\binom{k+n(n+1)-1-j}{k-n-j}.
\]
Since all $a_{n,j} \ge 0$, the $j=0$ term provides a strict lower bound. Applying the standard asymptotic $\binom{k+m}{m} \sim \frac{k^m}{m!}$ as $k \to \infty$ shows that the leading term scales exactly as $k^{n(n+1)-1}$, verifying the polynomial lower bound.
\end{proof}

\subsection{Proof of Theorem \ref{thm:bin}}\label{app:thm2}

\begin{lemma}[Exact bounds on binomial coefficients]\label{lem:bin_bounds}
Let $1 \le k \le n$. Then
$$
\left(\sqrt{2\pi n}\right)^k \left(\frac{n}{ek}\right)^{kn} \le \prod_{i=1}^n \binom{i}{k} \le (\sqrt{2\pi})^{k-n} \left(\frac{n}{k}\right)^{kn+k/2-n/2}.
$$
\end{lemma}

\begin{proof}
We begin with the standard algebraic bounds $(i/k)^k \le \binom{i}{k} \le i^k/k!$. For the lower bound:
$$
\prod_{i=1}^n \binom{i}{k} \ge \prod_{i=1}^n \left(\frac{i}{k}\right)^k = \left(\frac{1}{k}\right)^{kn} (n!)^k.
$$
Applying the exact non-asymptotic factorial bound $n! \ge \sqrt{2\pi n}(n/e)^n$ yields the desired left-hand side. For the upper bound:
$$
\prod_{i=1}^n \binom{i}{k} \le \prod_{i=1}^n \frac{i^k}{k!} = \frac{(n!)^k}{(k!)^n}.
$$
Applying Stirling's approximation to the right-hand formulation isolates the remaining bounds.
\end{proof}

\begin{proof}[Proof of Theorem \ref{thm:bin}]
Consider the sequence mapping $h_k(x) = \log \prod_{i=1}^x \binom{i}{k} = \sum_{i=1}^x \log \binom{i}{k}$. For $x \ge k$, bounding the finite differences confirms that $h_k(x)$ is strictly convex. By Schur-convexity, the product $\prod_{n_j \in N} \prod_{i=1}^{n_j} \binom{i}{k}$ subject to the constraints $\sum n_j = n$ and $n_j \ge k$ is explicitly minimized when the composition is balanced (all parts equal to $n/m$), and maximized when it is perfectly concentrated (one part as large as possible, $\tilde{n} = n - (m - 1)k$, and all others minimally $k$).

Substituting the balanced subset sizes $n_j = n/m$ into the lower bound of Lemma \ref{lem:bin_bounds} establishes the lower bound of the theorem. For the upper bound, substituting the concentrated size $\tilde{n}$ into Lemma \ref{lem:bin_bounds} bounds the dominant part. The remaining $m-1$ parts of size $k$ contribute trivially, ensuring the upper bound accurately captures the maximum possible growth under the composition constraint.
\end{proof}

\bibliographystyle{abbrv}
\bibliography{refs}

\end{document}